\documentclass[11pt,english]{article}

\usepackage[margin = 2.50cm]{geometry}

\usepackage{amsthm}
\usepackage{amsmath}
\usepackage{amssymb}
\usepackage{setspace}
\usepackage{mathtools}
\usepackage{graphicx}
\usepackage[hidelinks]{hyperref}
\usepackage{thmtools}
\usepackage{thm-restate}
\usepackage{cleveref}

\usepackage{enumerate}
\usepackage{framed}
\usepackage{subcaption}

\usepackage{floatrow}
\usepackage{comment}
\theoremstyle{plain}

\newtheorem*{thm*}{Theorem}
\newtheorem{thm}{Theorem}
\Crefname{thm}{Theorem}{Theorems}

\newtheorem*{lem*}{Lemma}
\newtheorem{lem}[thm]{Lemma}
\Crefname{lem}{Lemma}{Lemmas}

\newtheorem*{claim*}{Claim}
\newtheorem{claim}[thm]{Claim}
\crefname{claim}{Claim}{Claims}
\Crefname{claim}{Claim}{Claims}

\Crefname{prop}{Proposition}{Propositions}

\newtheorem{cor}[thm]{Corollary}
\crefname{cor}{Corollary}{Corollaries}

\crefname{conj}{Conjecture}{Conjectures}

\Crefname{qn}{Question}{Questions}

\newtheorem{obs}[thm]{Observation}
\Crefname{obs}{Observation}{Observations}

\Crefname{ex}{Example}{Examples}

\theoremstyle{definition}

\Crefname{prob}{Problem}{Problems}

\newtheorem{defn}[thm]{Definition}
\Crefname{defn}{Definition}{Definitions}

\theoremstyle{remark}

\renewenvironment{proof}[1][]{\begin{trivlist}
\item[\hspace{\labelsep}{\bf\noindent Proof#1.\/}] }{\qed\end{trivlist}}

\newcommand{\remove}[1]{}

\newcommand{\eps}{\varepsilon}
\newcommand{\whp}{with high probability}
\newcommand{\pr}{\mathbb{P}}
\newcommand{\dirpath}[1]{\overrightarrow{P_{#1}}}

\AtBeginDocument{%
   \def\MR#1{}
}

\begin{document}


\title{Monochromatic paths in random tournaments}
\date{\vspace{-5ex}}
\author{
	Matija Buci\'c\thanks{
	    Department of Mathematics, 
	    ETH, 
	    8092 Zurich;
	    e-mail: \texttt{matija.bucic}@\texttt{math.ethz.ch}.
	}
	\and
    Shoham Letzter\thanks{
        ETH Institute for Theoretical Studies,
        ETH,
        8092 Zurich;
        e-mail: \texttt{shoham.letzter}@\texttt{eth-its.ethz.ch}.
    }
    \and
	Benny Sudakov\thanks{
	    Department of Mathematics, 
	    ETH, 
	    8092 Zurich;
	    e-mail: \texttt{benjamin.sudakov}@\texttt{math.ethz.ch}.     
	}
}

\maketitle

\begin{abstract}

    \setlength{\parskip}{\medskipamount}
    \setlength{\parindent}{0pt}
    \noindent
    
    We prove that, \whp, any $2$-edge-colouring of a random tournament on $n$ vertices contains a monochromatic path of length $\Omega(n / \sqrt{\log n})$. This resolves a conjecture of Ben-Eliezer, Krivelevich and Sudakov and implies a nearly tight upper bound on the oriented size Ramsey number of a directed path.

\end{abstract}
\section{Introduction} \label{sec:intro}

Ramsey theory refers to a large body of mathematical results, which roughly say that any sufficiently large structure is guaranteed to have a large well-organised substructure. For example, the celebrated theorem of Ramsey \cite{ramsey1929problem} says that for any fixed graph $H$, every $2$-edge-colouring of a sufficiently large complete graph contains a monochromatic copy of $H$. The \emph{Ramsey number of $H$} is defined to be the smallest order of a complete graph satisfying this property.

In this paper, we study an analogous phenomenon for oriented graphs. An \emph{oriented graph} is a directed graph which can be obtained from
a simple undirected graph by orienting its edges. The undirected graph, obtained by ignoring edge orientations of an oriented graph, is called its \textit{underlying graph}.

A \emph{tournament} is an orientation of the complete graph.
Given directed graphs $G$ and $H$, we write $G \to H$ if any $2$-edge-colouring of $G$ contains a monochromatic copy of $H$. The \textit{oriented Ramsey number} of $H$ is defined to be the smallest $n$ for which $T \to H$ for every tournament $T$ on $n$ vertices.

While in the undirected case the Ramsey number exists and is finite for every $H$, in the directed case the Ramsey number can be infinite.
Indeed, given a directed graph $H$ that contains a directed cycle and any tournament $T$, we colour the edges of $T$ as follows. Denote the vertices of $T$ by $v_1, \ldots, v_n$ and colour edges $\overrightarrow{v_i v_j}$ red if $i<j$ and blue otherwise. It is easy to see that this colouring has no monochromatic directed cycle, so, in particular, there is no monochromatic copy of $H$.

Therefore it only makes sense to study the Ramsey properties for acyclic graphs. We note that the oriented Ramsey number of an acyclic graph is always finite. This follows since every acyclic graph is a subgraph of a transitive tournament, and the oriented Ramsey number of transitive tournaments is finite by essentially the same argument as for the complete graphs in the undirected Ramsey case. One of the most basic classes of acyclic graphs are oriented paths and in particular, directed paths (which are oriented paths whose edges are all oriented in the same direction). In this paper we study the problem of finding such monochromatic paths in tournaments. 

Another major difference between the undirected and the oriented Ramsey numbers is the fact that there is only one complete graph on $n$ vertices, while there are many tournaments on $n$ vertices. In particular, the answer to how long a monochromatic path we can find in a tournament on $n$ vertices $T$, depends on $T$ as well as the colouring of the edges. 
We thus define $m(T)$ to be the largest $m$ such that $T \to \dirpath{m}$, where $\dirpath{m}$ is the directed path on $m$ vertices. 

The celebrated Gallai-Hasse-Roy-Vitaver theorem \cite{gallai1968directed,hasse1965algebraischen,roy1967nombre,vitaver1962determination}, states that any directed graph, whose underlying graph has chromatic number at least $n$, contains a directed path on $n$ vertices. A simple consequence of this theorem is that $m(T) \ge \sqrt{n}$, for any tournament $T$ on $n$ vertices. Indeed, let $T$ be a $2$-coloured tournament on $n$ vertices, let $T_R$ be the subgraph of $T$ consisting of red edges and $T_B$ the subgraph consisting of blue edges. Since $\chi(T)=n$ we have $\chi(T_R) \ge \sqrt{n}$ or $\chi(T_B) \ge \sqrt{n}$ (by $\chi(G)$, where $G$ is an oriented graph, we mean the chromatic number of its underlying graph). In either case the aforementioned theorem implies there is a monochromatic path on at least $\sqrt{n}$ vertices.

To see that this is tight, consider the following $2$-colouring of a \textit{transitive tournament} $T$ on $n$ vertices, whose vertices are $ 1, \ldots, n$ and $\overrightarrow{xy}$ is an edge if and only if $x<y$; for convenience we assume that $n$ is a perfect square. Partition the vertices of $T$ into subsets $A_i$ of size $\sqrt{n}$, such that all vertices in $A_i$ are smaller than all vertices in $A_j$ for $i<j$. Colour the edges within $A_i$ blue and the other edges red. It is easy to see that a monochromatic path in this colouring has at most $\sqrt{n}$ vertices. In particular $m(T) \le \sqrt{n}$, so $m(T) = \sqrt{n}$.

The above resolves the problem of minimising the value of $m(T)$ among all $n$-vertex tournaments. It is natural to consider the opposite question: what is the \emph{maximum} of $m(T)$ and which $T$ attains this maximum? 

This question was implicitly posed by Ben-Eliezer, Krivelevich and Sudakov in \cite{ben2012size}, where they note that $m(T) \le \frac{2n}{\sqrt{\log n}}$ for any $n$-vertex tournament. Indeed, it is well known and easy to see that any tournament of order $n$ has a transitive subtournament of order $\log n$. Using this we can find a sequence of vertex-disjoint transitive subtournaments $A_i$ of order $\frac{1}{2} \log n$ covering all but at most $\sqrt{n}$ vertices. We denote the set of remaining vertices by $B$. We now $2$-colour each of $A_i$, as described above, to ensure that the longest monochromatic path within $A_i$ is of length $\sqrt{|A_i|}$. We then colour all edges from $A_i$ to $A_j$ blue if $i<j$ and red if $i > j$. Finally, we colour the edges going out of $B$ red and edges going into $B$ blue, while edges within $B$ can be coloured arbitrarily. In this colouring, the longest monochromatic path has length at most $\frac{2n}{\log n} \sqrt{\frac{ \log n  }{2}}+\sqrt{n} \le \frac{2n}{\sqrt{\log n}}.$

Are there tournaments that achieve this bound? Intuitively, such tournaments should be `far away from being transitive'. With this in mind, Ben-Eliezer, Krivelevich and Sudakov \cite{ben2012size} conjectured that random tournaments achieve this bound, up to a constant factor, \whp.

The random tournament on $n$ vertices, which we denote by $T_n$, is obtained by orienting each edge uniformly at random and independently of other edges.
With respect to this probability space we say an event $A$ holds \emph{\whp{}} if the probability that $A$ occurs tends to $1$ as $n \to \infty$.

To support their conjecture, Ben-Eliezer, Krivelevich and Sudakov show that a random tournament satisfies $m(T_n) \ge \frac{cn}{\log n}$ \whp, where $c > 0$ is a constant. 
Here we improve this bound, proving the following essentially tight result.

\begin{restatable}{thm}{main} \label{thm:main}
There is a constant $c > 0$ such that a random tournament $T = T_n$ satisfies $T \rightarrow \dirpath{m}$, \whp{}, where $m=\frac{cn}{\sqrt{\log n}}$.
\end{restatable}

Theorem \ref{thm:main} resolves the question of the maximal value of $m(T_n)$, up to a constant, with respect to a directed path. This gives rise to a natural next question: is it possible to generalise it to other graphs, in place of the directed path? In this direction, we prove a generalisation of \Cref{thm:main} for arbitrarily oriented paths, which is tight up to a constant factor. Given an arbitrarily oriented path $P$, we denote by $\ell(P)$ the length of the longest directed subpath of $P$.

\begin{thm} \label{thm:oriented-paths}
    There is a constant $c>0$ such that, \whp{}, if $T$ is a random tournament on $c(\ell\sqrt{\log{\ell}} + n)$ vertices, then $T \rightarrow P$, for every oriented path $P$ on $n$ vertices with $\ell(P) \le \ell$.
\end{thm}

\Cref{thm:main} is related to the study of size Ramsey numbers, which was initiated by Erd\H{o}s, Faudree, Rousseau and Schelp \cite{erdHos1978size} in $1972$, and has received considerable attention since then. The \emph{size Ramsey number} of a graph $H$ is defined to be the minimum of $e(G)$ (the number of edges of $G$) over all graphs $G$ for which $G \rightarrow H$. Answering a question of Erd\H{o}s, Beck \cite{beck1983size} proved in $1983$ that the size Ramsey number of the undirected path on $n$ vertices is linear in $n$.

Here we consider the following natural analogue for oriented paths. 
Given an oriented graph $H$, we define the \emph{oriented size Ramsey number}, denoted by $r_e(H)$, to be  
$$r_e(H)= \min\{e(G) : G \to H, \text{ $G$ is an oriented graph} \}.$$

Ben-Eliezer, Krivelevich and Sudakov \cite{ben2012size} studied this number for the directed path, obtaining the following estimates, for some constants $c_1,c_2$.
$$ c_1 \frac{n^2 \log n}{ (\log \log n)^3}\le r_e\!\left(\dirpath{n}\right) \le c_2n^2 (\log n)^2.$$
We note that this illustrates the difference between the oriented and directed size Ramsey numbers (where in the directed case, $G$ is allowed to be any directed graph). Indeed, the directed size Ramsey number of $P_n$ is $\Theta(n^2)$, as follows from results by Raynaud \cite{raynaud1973circuit} (see also \cite{gyarfas1983vertex} for a simple proof) and Reimer \cite{reimer2002ramsey}.

Note that Theorem \ref{thm:main} immediately implies the following upper bound on $r_e\!\left(\dirpath{n}\right)$, which is tight up to a factor of $O((\log \log n)^3)$.

\begin{cor}\label{cor:size-ramsey}
There is a constant $c>0$ such that
$$r_e\!\left(\dirpath{n}\right) \le cn^2 \log n.$$
\end{cor}

\subsection{Organisation of the paper}
   In the following section, we list several results that we shall need for the proof of \Cref{thm:main}. 
   In \Cref{sec:main-proof}, we prove an asymmetric version of \Cref{thm:main} that applies to pseudorandom tournaments from which we easily deduce \Cref{thm:main}. We prove \Cref{thm:oriented-paths} in \Cref{sec:oriented-paths}. We conclude the paper in \Cref{sec:conclusion} with some remarks and open problems.
   
   Throughout the paper we do not try to optimize constants, all logarithms are in base $2$ and we omit floor and ceiling signs whenever these are not crucial. By $2$-colouring, we always mean a colouring of edges in two colours.
 
\section{Preliminaries} \label{sec:prelims}
In this section we list several general results we need for the proof of Theorem \ref{thm:main}. We start with a simple observation.

\begin{obs} \label{obs:min-deg-cycle}
    Let $G$ be a directed graph in which every vertex has in-degree at least $k$. Then $G$ contains a directed cycle of length at least $k+1$.
\end{obs}

\begin{proof}
    Consider a longest directed path $v_1 \ldots v_l$ in $G$. Note that $v_1$ cannot have any in-neighbours outside of the path, by the maximality assumption on the path. But, as $v_1$ has in-degree at least $k$, $v_1$ needs to have an in-neighbour $v_i$ on the path with $i \ge k+1$. Now, $(v_1 \ldots v_i)$ is a directed cycle of length at least $k+1$, as required.
\end{proof}

We give several simple definitions.
Given two subsets of vertices of a directed graph $G$, let $e_G(A,B)$ denote the number of edges of $G$ oriented from $A$ to $B$; we often omit $G$ and write $e(A,B)$ when it is clear which $G$ it refers to. If $G$ is a $2$-coloured graph, we denote by $e_R(A, B)$ the number of red edges from $A$ to $B$, and by $e_B(A, B)$ the number of blue edges from $A$ to $B$.

\begin{defn}
An oriented graph is said to be $(\eps, k)$\textit{-pseudorandom} if for every two disjoint sets $A$ and $B$ of size at least $k$, we have $e(A,B) \ge \eps|A||B|$.
\end{defn}

Our proof of Theorem \ref{thm:main} works for any $(\eps, \sigma)$-pseudorandom graph, in place of a random tournament. However, we note that our definition of pseudorandomness does not coincide with the one used by Ben-Eliezer, Krivelevich and Sudakov in \cite{ben2012size}. In particular they defined a graph to be $k$-pseudorandom if between any two sets $A,B$ of size $k$ there is an edge oriented from $A$ to $B$. 

The following Lemma shows that a random tournament is $(\eps, \sigma \log n)$-pseudorandom with arbitrary $\eps<\frac{1}{2}$ and $\sigma = \sigma(\eps)$. This is the only property of random tournaments that we will use in the proof.

\begin{lem} \label{lem:pseudorandom}
    For any $0 < \eps < \frac{1}{2}$ there is a constant $\sigma = \sigma(\eps)$ such that if $T=T_n$, then \whp{} $T$ is $(\eps, \sigma \log n)$-pseudorandom. 
\end{lem}

\begin{proof}
    Let $\delta = \frac{1}{2}-\eps.$ Let $A,B$ be a fixed pair of disjoint sets of vertices of size $k = \sigma \log n$, for a constant $\sigma$. Each of the $k^2$ possible edges is oriented from $A$ to $B$ with probability $1/2$, independently of all other edges, so if $X$ denotes the number of such edges, it is a binomially distributed random variable ($X \sim \text{Bin}\left(k^2, \frac{1}{2}\right)$).
    
    As such, Chernoff bounds (see appendix A of \cite{alon-spencer}) imply the following.
    $$\pr\left(X < \left(\frac{1}{2}-\delta \right)k^2\right) = \pr\left(X > \left(\frac{1}{2}+\delta \right)k^2\right) \le e^{-2\delta^2 k^2}.$$
    Now as there are at most $\binom{n}{k}^2$ choices for $A$ and $B$, the probability that there is a pair of subsets $A,B$ having $e(A,B)<\left(\frac{1}{2}-\delta \right)k^2$ is, by the union bound, at most:
    $$\binom{n}{k}^2 e^{-2\delta^2k^2} \le e^{2k \log n - 2 \delta^2 k^2}= e^{2\sigma \log^2 n (1 - \delta^2 \sigma)}=o(1),$$
    given that $\sigma > \delta^{-2}$.
    
    Finally, given two sets of size $a=|A| \ge k$ and $b=|B| \ge k$, by double counting the number of pairs of subsets of size $k$ a given edge is in, we get that $e(A,B) \ge \frac{\binom{a}{k}\binom{b}{k}\left(\frac{1}{2}-\delta\right)k^2}{\binom{a-1}{k-1}\binom{b-1}{k-1}} \ge \left(\frac{1}{2}-\delta\right)ab$.
    
    So, by setting $\sigma=2\left(\frac{1}{2}-\eps\right)^{-2}$, we see that $T$ is $(\eps, \sigma \log n)$-pseudorandom \whp.
\end{proof}

The following lemma is due to Ben-Eliezer, Krivelevich and Sudakov \cite{ben2012long-cycles,ben2012size} and is a useful application of the Depth-first search (DFS) algorithm.

\begin{lem}\label{lem:dfs}
    Given an oriented graph $G$, there is a directed path $P$ such that vertices of $G-P$ can be partitioned into two disjoint sets $U$ and $W$ such that $|U|=|W|$ and all the edges between $U$ and $W$ are oriented from $W$ to $U$.
\end{lem}

\begin{proof}
    We start with $U=\emptyset$, $W=V(G)$ and $P= \emptyset.$ We repeat the following procedure, throughout of which $P$ is a path, $|W| \ge |U|$ and all edges between $U$ and $W$ being oriented from $W$ to $U$. If $P$ is empty we take any vertex from $W$ to be the new path $P$ and remove the vertex from $W$. Otherwise, we consider the endpoint $v$ of $P$; if $v$ has any out-neighbours in $W$ we use one of them to extend $P$ and remove it from $W$, and otherwise we remove $v$ from $P$ and add it to $U$.
    Note that at each step either $|W|$ decreases by one or $|U|$ increases by $1$ so the value $|W|-|U|$ decreases by exactly $1$ per step, in particular it reaches $0$, at which point we stop. At this point, $P$, $U$ and $W$ satisfy the requirements of the lemma.
\end{proof}

The following lemma will prove useful at several places in the proof of \Cref{thm:main}. It will mostly be used by exploiting pseudorandomness to imply the conditions of the lemma are satisfied.
\begin{lem}\label{lem:const-cycle}
    Let $\eps>0$, and let $A$ and $B$ be disjoint sets of $n$ vertices of a directed graph $G$. Suppose that $e(B, A) \ge \eps n^2$, and $e(X,Y) > 0$ for every two subsets $X \subseteq A$ and $Y \subseteq B$ of size at least $\frac{\eps}{8} n$. Then there is a directed cycle, alternating between $A$ and $B$, whose length is at least $\frac{\eps}{4}n$.
\end{lem}
\begin{proof}

    Consider the graph $H$, consisting only of edges between $A$ and $B$. We show that $H$ contains a cycle of length at least $\frac{\eps}{4}n$. Note that it suffices to show that there is a path $P$ in $H$, consisting of at least $\frac{3\eps}{4} n$ vertices. Indeed, let $Y$ be the set of the first $\frac{\eps}{8}n$ vertices of $P$ that are in $B$ and let $X$ be the set of the last $\frac{\eps}{8}n$ vertices of $P$ in $A$. As vertices of $P$ alternate between $A$ and $B$, $Y$ is in the first third of the path while $X$ is in the last third. By the assumptions, there is an edge from $X$ to $Y$, which completes a cycle containing the middle third of $P$, so has at least the required length.

    The vertices of $H$ may be partitioned into a directed path $P$ and two sets $U$ and $W$ of equal size such that there are no edges from $U$ to $W$, by Lemma \ref{lem:dfs} applied to $H$. By the above, we may assume that $P$ has at most $\frac{3\eps}{8}n$ vertices in each of $A$ and $B$. Write $U_A = U \cap A$ and similarly for $U_B, W_A, W_B, P_A, P_B$. Then $|U|=|W|$ and $|A|=|B|$ imply that $|P|$ is even, which in turn due to $P$ alternating between $A$ and $B$ implies $|P_A|=|P_B|$. Combining these equalities, we obtain $|U_A| = |W_B|$ and $|U_B| = |W_A|$. Note that, as there are no edges from $U_A$ to $W_B$, by the assumptions we have $|U_A| < \frac{\eps}{8}n$. So the edges from $B$ to $A$ can be covered by $V(P) \cup U_A \cup W_B$, which is a set of size smaller than $\eps n$, so there are fewer than $\eps n^2$ such edges, a contradiction to the assumption that $e(B,A) \ge \eps n^2$.
\end{proof}

The following theorem is a result of Raynaud \cite{raynaud1973circuit} (see also \cite{gyarfas1983vertex} for a simple proof) about directed monochromatic paths in $2$-colourings of the complete directed graph.  
\begin{thm}[Raynaud \cite{raynaud1973circuit}] \label{thm:raynaud}
In every $2$-colouring of the complete directed graph on $n$ vertices, there is a directed monochromatic path of length $\frac{n}{2}$.
\end{thm}

We restate the Gallai-Hasse-Roy-Vitaver theorem mentioned in the introduction.
\begin{thm} \label{thm:ghrv} 
    Let $G$ be a directed graph, whose underlying graph has chromatic number at least $n$. Then $G$ contains a directed path of length $n-1$.
\end{thm}

The following is an asymmetric variant of a result mentioned in the introduction; its proof is again a simple consequence of \Cref{thm:ghrv}, so we omit the details.
\begin{cor}\label{cor:asym-tour-paths} 
    In every $2$-colouring of a tournament on at least $xy+1$ vertices, we can either find a red path of length $x$ or a blue path of length $y$.
\end{cor}

\section{Main result} \label{sec:main-proof}

Given directed graphs $G$, $H_1$ and $H_2$, we write $G \to (H_1,H_2)$ if in every $2$-colouring of $G$ there is a blue copy of $H_1$ or a red copy of $H_2$. Note that $G \to (H,H)$ is same as $G \to H$.

The following theorem is a generalisation of \Cref{thm:main}, that is applicable for pseudorandom graphs, and treats the asymmetric case. The statement for pseudorandom graphs is convenient for the proof of \Cref{thm:oriented-paths} which we give in the next section.

\begin{thm}\label{thm:main-general}
Given $0< \eps < \frac{1}{2}$ and $0 < \sigma$, let $G$ be an $(\eps, \sigma \log n)$-pseudorandom tournament on $n$ vertices. There is a constant $c>0$, such that $G \to \left(\dirpath{r}, \dirpath{s}\right)$ provided $r,s \le cn$ and $rs \le \frac{cn^2}{ \log n}$.
\end{thm}

We note that, by setting $r = s = \frac{cn}{\sqrt{\log n}}$ and combining with Lemma \ref{lem:pseudorandom}, we obtain Theorem \ref{thm:main}, while, by taking $r = s = \frac{cn}{\log n}$, we obtain the previously best result of Ben-Eliezer, Krivelevich and Sudakov \cite{ben2012size} (in fact, they proved the special case of \Cref{thm:main-general} where $r = cn$ and $s = \frac{cn}{\log n}$). 
   
\begin{proof}

Let us fix a colouring of the edges of $G$. The proof differs depending on the monochromatic cycle structure of $G$. Let $a=24\eps^{-1}\sigma$ and $b=12\eps^{-1}a$; we say that 
$$
    \text{a cycle } C \text{ is } 
        \left\{
        \begin{array}{ll}
        \text{\textit{short}} & \text{if }\, |C| < a \log n,\\
        \text{\textit{medium}} & \text{if }\, a \log n \le |C| \le b \log n,\\
        \text{\textit{long}} & \text{if }\,  b \log n < |C|.
        \end{array} 
        \right.
$$
Our proof is divided into two parts. In the first part we consider the case where there are many vertex disjoint monochromatic medium cycles, whereas in the second part we consider the case where there are many vertices that are not contained in any medium monochromatic cycles.

\subsection*{Case 1: many disjoint medium monochromatic cycles} 

    Suppose that there is a collection of vertex-disjoint medium monochromatic cycles that cover at least $n/2$ vertices. Then, without loss of generality, there are pairwise disjoint medium blue cycles $C_1 \ldots, C_t$, that cover at least $\frac{n}{4}$ vertices. Note that $t \ge \frac{n}{4b \log n} = \frac{n}{48 \eps^{-1}a \log n}$, by the upper bound on the length of the cycles. We will show that in this case there is a monochromatic path of length linear in $n$.

    We define an auxiliary $2$-colouring $H$ of the complete directed graph on vertex set $[t]$ as follows. The edge $ij$ is blue if at least $\frac{a}{4}\log n$ vertices in $C_i$ have a blue out-neighbour in $C_j$; otherwise, $ij$ is red. 
    
    Let $M$ be a maximum \emph{red-red} matching in $H$, i.e.~$M$ is a matching whose edges are coloured red in both directions; we call such edges \textit{red-red}. 

    Suppose first that $M$ consists of at most $\frac{t}{4}$ edges. Then, by maximality of $M$, the set of vertices not in $M$ spans no red-red edges. In particular, $H$ contains a blue tournament $T'$ of order $t/2$. 
    As any tournament contains a Hamilton path there is a blue directed path of order $t/2$ in $H$. Let $B_1, \ldots, B_{t/2}$ be the blue cycles corresponding to the vertices in this path, preserving the order of the path. We use the following claim to find a long blue path.
    
    \begin{claim}\label{claim:join-paths}
    Let $C_1, \ldots, C_k$ be a sequence of pairwise vertex-disjoint directed cycles in a directed graph $G$, such that for each $i<k$, at least $r$ vertices of $C_i$ have an out-neighbour in $C_{i+1}$. Then there is a path of length at least $(k-1)r$ in $G$.
    \end{claim}
    \begin{proof}
        We start the path at an arbitrary vertex $v_1$ of $C_1$, we follow $C_1$ to the last vertex that sends an edge towards $C_{2}$, and then follow such an edge to a vertex $v_2 \in V(C_2)$. We repeat this for $i<k$: starting from $v_i$ we follow $C_i$ until the last vertex that sends an edge towards $C_{i+1}$ and then follow that edge to obtain $v_{i+1}$. By the assumption, there are at least $r$ vertices with an edge toward $C_{i+1}$ so the part of the path from $v_i$ to $v_{i+1}$ is of length at least $r$. Hence, the path we obtain has the desired length.
    \end{proof}

    Using Claim \ref{claim:join-paths} with the blue cycles $B_i$ and blue edges between them, we find a monochromatic path of length at least $\left(\frac{t}{2}-1\right) \cdot \frac{a}{4}\log n \ge cn$, provided $c$ is small enough. 

    Now suppose that $M$ consists of at least $t/4$ edges. We use the following claim to find red cycles whose intersection with existing blue cycles $C_i$ is large. Then we repeat the above argument for intersections of red and blue cycles, with the added benefit that now the setting is more symmetric. 
    
    \begin{claim}\label{claim:red-red}
        If both $\overrightarrow{ij}$ and $\overrightarrow{ji}$ are red edges of $H$, then there is a red cycle in $V(C_i) \cup V(C_j)$ whose intersection with $C_i$ has size at least $\frac{3a \eps}{32} \log n$.
    \end{claim}
    \begin{proof}
        Since $\overrightarrow{ij},\overrightarrow{ji}$ are red, there are subsets $A \subseteq C_i$ and $B \subseteq C_j$, of size $\frac{3a}{4}\log n$, with no blue edges between $A$ and $B$. By pseudorandomness, $e_R(B, A) \ge \eps |A||B|$, and for any subsets $X \subseteq A, Y \subseteq B$ of sizes at least $\frac{\eps |A|}{8}$ we have $e_R(X, Y) \ge \eps |X| |Y|$ (as $\frac{3a\eps}{32} \ge \sigma$). By Lemma \ref{lem:const-cycle}, there is a red cycle, alternating between $A$ and $B$, of length at least $\frac{3a \eps}{16} \log n$, so its intersection with $C_i$ has the required size.
    \end{proof}
    For each edge in $M$ we apply Claim \ref{claim:red-red} to obtain a collection of disjoint medium blue cycles $B_1,\ldots,B_{k}$ and disjoint red cycles $R_1, \ldots, R_{k}$, with $k \ge \frac{t}{4}$, such that $|I_i| \ge \frac{3a \eps}{32} \log n \ge 2\sigma \log n$, where $I_i=B_i \cap R_i$.

    We now define an auxiliary $2$-colouring of the complete directed graph on vertex set $[k]$, similarly to the one given above: here $\overrightarrow{ij}$ is blue if at least $|I_i|/4$ vertices of $I_i$ send a blue edge towards $I_j$, and red otherwise. Note that when $\overrightarrow{ij}$ is red, at least $\frac{|I_i|}{4}$  vertices of $I_i$ send a red edge towards $I_j$ (indeed, otherwise, we find a subset of $I_i$ of size at least $|I_i|/2 \ge \sigma \log n$ which sends no edges to $I_j$, a contradiction to pseudorandomness).
    
    By Theorem \ref{thm:raynaud}, there is a monochromatic path of length $\frac{k}{2}$ in the auxiliary graph. Now, by Claim \ref{claim:join-paths}, there is a monochromatic path of length at least $\frac{k}{2}\cdot \frac{\sigma}{2}\log n \ge cn$, provided $c$ is small enough. 
        
\subsection*{Case 2: there is a large set of vertices spanning no medium cycles}
    We now assume that there is a subset $U$ of at least $n/2$ vertices that spans no medium monochromatic cycles.
    
    We first show that, in fact, $U$ does not span long monochromatic cycles either.
    Suppose to the contrary that $C= (v_1 v_2 \ldots v_k)$ is a long monochromatic cycle in $U$ of minimum length; without loss of generality, $C$ is blue. Note that all chords of $C$ of length at least $|C|/6$ are red, as otherwise we can obtain a blue cycle of length at least $|C|/6$ but less than $|C|$, a contradiction to the choice of $C$ or the assumption that there are no medium blue cycles in $U$. Let $A= \{v_1, \ldots, v_{k/3}\}$ and $B= \{v_{k/2 + 1}, \ldots, v_{5k/6}\}$. As explained above, all the edges between $A$ and $B$ are red. By Lemma 
    \ref{lem:const-cycle} (which can be applied due to pseudorandomness as $\frac{\eps |A|}{8} \ge \frac{\eps b}{24} \log n>\sigma \log n$), $A \cup B$ spans a red cycle $C$ of length at least $\frac{\eps |A|}{4} \ge \frac{\eps b}{12} \log n = a \log n$, so $C$ is either a medium cycle or a shorter monochromatic long cycle. Either way, we reach a contradiction, hence there are no monochromatic long cycles.
    
    Let $m = |U|$.
    We note that the fact that there are only short blue cycles implies the existence of an ordering of the vertices with few back blue edges.
        
    \begin{claim}\label{claim:order}
        There is an order of the vertices $u_1, \ldots, u_m$ such that for every $i$ there are at most $a \log n$ indices $j>i$ such that the edge $\overrightarrow{u_j u_i}$ exists and is blue.
    \end{claim}
    
    \begin{proof}
        If there is a subtournament of $T$ in which minimum blue in-degree is of size at least $a \log n$ then, by \Cref{obs:min-deg-cycle}, we can find a blue cycle of length at least $a \log n$, a contradiction. 
        Hence, there is a vertex $u_1$ in $U$ whose in-degree is at most $a \log n$.
        Similarly, if $u_1, \ldots, u_{i-1}$ are defined, we may take $u_i$ to be a vertex with in-degree at most $a \log n$ in $T_i = T \setminus \{u_1,\ldots,u_{i-1}\}$. The order $u_1, \ldots, u_m$ satisfies the requirements of the claim.
    \end{proof}
    
    We now take $d=120 \eps^{-2} \sigma$ and $k=10d \log n$, and define sets $U_1, \ldots, U_{\frac{n}{2k}}$ by $U_i \equiv \{u_{(i-1)k+1}, \ldots, u_{ik}\}$. Note that for each $i$ we have $|U_i|=k$.
    
    \begin{claim}\label{claim:two-direction-pair}
        Let $i < j$, and let $W_i$ and $W_j$ be subsets of $U_i$ and $U_j$, respectively, of size $k/10$. Then there is a a red edge from $W_j$ to $W_i$ and a blue edge from $W_i$ to $W_j$.
    \end{claim}
    \begin{proof}
        Let $W_i$ and $W_j$ be as in the claim.
        We start with the red edges. By pseudorandomness, as $|W_i| = |W_j| = \frac{k}{10} = d \log n \ge \sigma \log n$, we have $e(W_j, W_i) \ge \eps |W_j||W_i|.$ Also, by the property of the order given by Claim \ref{claim:order}, at most $|W_i| a \log n$ of the edges from $W_j$ to $W_i$ are blue. It follows that $e_R(W_j,W_i) \ge \eps|W_j||W_i|-|W_i|a \log n = |W_j||W_i|\left(\eps - \frac{a \log n}{|W_j|}\right) \ge \left(\eps - \frac{24 \eps^{-1}\sigma \log n}{120 \eps^{-2}\sigma \log n}\right)|W_j||W_i| = \frac{4\eps}{5}|W_j||W_i|>0,$ as desired.

        For the blue edges, let us assume, contrary to the statement of the claim, that all the edges from $W_i$ to $W_j$ are red. By the previous paragraph $e_R(W_j,W_i) \ge \frac{4\eps}{5} |W_i||W_j|$, while for any subsets $X \subseteq W_i,Y \subseteq W_j$ of size at least $\frac{\eps}{10}|W_i| \ge \sigma \log n$ we have $e_R(X,Y) \ge \eps |X||Y|>0$. Therefore, by Lemma \ref{lem:const-cycle} with parameter $\frac{4}{5}\eps$, there is a red cycle of length at least $\frac{\eps}{5}|W_i| = 24\eps^{-1}\sigma \log n = a \log n$, which is a contradiction to the assumption that there are only short cycles.
    \end{proof}
    
    Recall that we are looking for a blue directed path of length $r$ or a red directed path of length $s$. Let $x=\frac{4rk}{n},y=\frac{4sk}{n}.$ We say that a monochromatic path is \textit{long} if it is blue and of length at least $x$ or if it is red and of length at least $y$.
    
    We note that in any tournament of order $\frac{k}{5} \ge \sigma \log n $ we can find a long path. Indeed, this follows from  \Cref{cor:asym-tour-paths} as $xy+1=\frac{(40d)^2rs \log^2 n}{n^2} +1\le (40d)^2c \log n +1 \le \sigma \log n$, provided $c$ is small enough.
    
    Within each $U_i$ we repeat the following: we find a long monochromatic path and remove its start and end vertices, then repeat with the remaining graph as long as we can find a long monochromatic path. Note that we remove only two vertices per round, and as long as we have at least $\frac{k}{5}$ vertices left, we can continue, so this process runs for at least $\frac{2k}{5}$ rounds. In particular, we find at least $\frac{2k}{5}$ corresponding start and end points of long paths within $U_i$. 
        
    We call $U_i$ red if we found more red long paths, and blue otherwise. Without loss of generality, there is a collection of $\frac{n}{4k}$ red sets $U_i$, which we denote by $R_1, \ldots, R_{\frac{n}{4k}}$, preserving the order of the $U_i$'s. 
        
    Denote the set of start points in $R_i$ by $S_i$ and the set of endpoints by $E_i$. Then, by construction, $S_i$ and $E_i$ are disjoint and have size at least $\frac{k}{5}$.
    Let $X_i$ be the subset of vertices of $S_i$ that are starting points of a red path of length at least $iy$ contained in $R_1 \cup \ldots \cup R_i$. 
        
    \begin{claim}\label{claim:startpoints} 
        For any $i \le \frac{n}{4k}$, we have $|X_i| \ge \frac{k}{10}$.
    \end{claim}
    
    \begin{proof}
        We prove the claim inductively. For the basis, $X_1 = S_1$, so $|X_1| \ge \frac{k}{5}$. We assume the claim is true for $i$, so $|X_i| \ge \frac{k}{10}$. 
    
        Consider the set $Y_{i+1}$ of vertices of $E_{i+1}$ with a red edge towards $X_i$. Given $v \in Y_{i+1}$, let $u \in S_{i+1}$ be the corresponding start point of a red long path in $R_{i+1}$, and let $w$ be a red out-neighbour of $v$ in $X_i$. By taking the red path of length $y$ in $R_{i+1}$ from $u$ to $v$, appending to it the edge $vw$ and a path of length $iy$ starting at $w$, given by the inductive assumption, we find a path of length $(i+1)y$ starting at $u$, hence $u \in X_{i+1}$. 
        
        If $ |Y_{i+1}| \le \frac{k}{10}$, then $X_i$ and $E_{i+1} \setminus Y_{i+1}$ are both of size at least $\frac{k}{10}$, but there are no red edges from the second set to the first, a contradiction to Claim \ref{claim:two-direction-pair}. It follows that $|Y_{i+1}| \ge \frac{k}{10}$, so by taking the start points corresponding to vertices in $Y_{i+1}$, we find at least $\frac{k}{10}$ vertices in $X_{i+1}$, completing the proof.
    \end{proof}
        
    The statement of Claim \ref{claim:startpoints}, for $i=\frac{n}{4k}$, implies that there is a red monochromatic path of length $\frac{n}{4k}y = s$, thus completing the proof of Theorem \ref{thm:main-general} in this case.
    
    To complete the proof, consider a maximal collection of disjoint monochromatic medium cycles. If it covers more than $\frac{n}{2}$ vertices we are done by Case 1. Otherwise, let $U$ be the set of remaining vertices, then $|U| \ge \frac{n}{2}$ and there are no medium cycles in $U$, and we are done by Case 2.
\end{proof}   
We note that in the above proof, by tracking the required constants we find that the value $c=\left(\frac{\eps^2}{4800 \sigma}\right)^2$ suffices. Using \Cref{lem:pseudorandom} and optimising the above expression over $\eps$ we obtain that $c=2^{21}$ suffices in the statement of \Cref{thm:main}.

\section{Arbitrarily oriented paths} \label{sec:oriented-paths}

We denote by $P_{n_1, \ldots, n_k}$ the oriented path consisting of $k$ maximal directed subpaths, with the $i$-th one of length $n_i$. Recall that $\ell(P)$ is the length of the longest directed subpath of an oriented path $P$. Our aim in this section is to extend \Cref{thm:main-general} to arbitrarily oriented paths. We prove the following result, note that \Cref{thm:oriented-paths} easily follows from this result, by \Cref{lem:pseudorandom}.

\begin{thm} \label{thm:oriented-paths-general}
    Given $\eps, \sigma > 0$, there is a constant $c$ such that the following holds. Let $T$ be an $(\eps, \sigma \log |T|)$-pseudorandom tournament on $c(\ell\sqrt{\log{\ell}} + n)$ vertices. Then $T \rightarrow P$, for every oriented path $P$ on $n$ vertices with $\ell(P) \le \ell$.
\end{thm}

\Cref{thm:oriented-paths-general} is tight, up to a constant factor. To see this, recall, from the tightness argument for \Cref{thm:main}, that any tournament of size at most $\frac{1}{2}\ell \sqrt{\log \ell}$ can be coloured without having a monochromatic $\dirpath{\ell},$ so certainly no copy of $P.$ In addition, as $|P|=n,$ any tournament on at most $n-1$ vertices does not contain $P$. Combining these two observations we conclude that any tournament of size $\frac{1}{4}(\ell \sqrt{\log \ell}+n)$ can be coloured without having a monochromatic copy of $P.$

\Cref{thm:oriented-paths-general} follows from \Cref{thm:main-general}, with some additional ideas. 
Before turning to the proof, we mention a few preliminaries. 

Let $A$ and $B$ be disjoint subsets of $V(G)$, where $G$ is an oriented graph. Let $G(A,B)$ be the subgraph consisting only of edges of $G$ oriented from $A$ to $B$. We say that $(A,B)$ is a \emph{$k$-mindegree pair} if the underlying graph of $G(A, B)$ has minimum degree at least $k$. The following simple lemma shows how to find $k$-mindegree pairs in graphs with sufficiently many edges.

\begin{lem} \label{lem:find-min-deg-pair}
    Let $G$ be a directed graph on $n$ vertices with at least $dn$ edges. Then $G$ contains a $d/4$-mindegree pair.
\end{lem}

\begin{proof}
    We note that there is a bipartition $\{X, Y\}$ of $V(G)$ with $e(X, Y) \ge dn/4$. Indeed, if $\{X, Y\}$ is a random bipartition (i.e.~vertices are put in $X$ with probability $1/2$ independently of other vertices), then the expected number of edges from $X$ to $Y$ is $e(G)/4 \ge dn/4$. Hence the required partition exists.
    
    Now consider the underlying graph of $G(X, Y)$. This graph has at least $dn/4$ edges. Now remove, one by one, vertices of degree less than $d/4$ until no such vertices remain. Note that fewer than $dn/4$ edges are removed in this process, which implies that not all vertices were removed, i.e.~the resulting graph has minimum degree at least $d/4$. Denote by $X'$ and $Y'$ the sets of vertices remaining in $X$ and $Y$ respectively. Then $(X', Y')$ is a $d/4$-mindegree pair.
\end{proof}

Given an oriented path $P=P_{n_1, \ldots, n_k},$ let $P_i$ denote the $i$-th maximal directed subpath of $P$. The following lemma is the main machinery that we shall need in the proof of \Cref{thm:oriented-paths-general}. 
\begin{lem}\label{lem:construct-or-path}
    Let $G$ be a directed graph. Suppose that $e(G) \ge 4(x + n)|G|$ and that every subset of at least $x$ vertices contains $\dirpath{\ell}$. Then any oriented path $P$ on $n$ vertices, with $l(P) \le l$, is a subgraph of $G$.
\end{lem}
    \begin{proof}
        By Lemma \ref{lem:find-min-deg-pair}, there is an $(x + n)$-mindegree pair $(A, B)$ in $G$.
        Let $P=P_{n_1, \ldots, n_k}.$ We prove by induction on $k$, that given the assumptions of the lemma, we can embed $P$ with last vertex $v$ in $A$ if $P_k$ is directed towards $v$, and in $B$ otherwise. For the basis, the case $k=1$, the choices of $A$ and $B$ implies that $|A|,|B| \ge x+n$ so within both $A$ and $B$ we can find a copy of $P_1$.
        
        Let us assume that the statement holds for paths consisting of $k-1$ maximal directed subpaths. Then we can find $P'=P_{n_1, \ldots, n_{k-1}-1}$ in $G$. Let $v$ be the last vertex of $P'$. We assume that $v$ is an end vertex of $P'_{k-1}$, so $v \in A$, the other case can be treated in a similar fashion.
        
        Let $S$ be the set of out-neighbours of $v$ in $B$ that are not in $P'$. As $v$ has at least $x+n$ out-neighbours in $B$, and $|P'|\le n$, we have $|S| \ge x$. By the assumption we can embed $P_k=\dirpath{n_k}$ in $S$, as $n_k \le \ell$; denote its end vertex by $u$. Now combining $P'$, the edge $vu$ and a copy of $P_k$ in $S$ starting with $u$, we obtain a copy of $P$ with the last vertex in $B$, and $P_k$ oriented away from it, as desired.
    \end{proof}
    
We are now ready to prove \Cref{thm:oriented-paths-general}.
\begin{proof} [ of \Cref{thm:oriented-paths-general}]
    Let $x = \eps 2^{-12}|T|$, we choose $c$ large enough in order for $x \ge n$ and $2 \sigma \log x \ge \sigma \log |T|$ to hold. We first note that every subset of at least $x$ vertices of $T$ contains a monochromatic $\dirpath{\ell}$, provided $c$ is large enough. Indeed, every induced subgraph $T'$ of $T$ on $x$ vertices is $(\eps, 2\sigma \log x)$-pseudorandom. By Theorem \ref{thm:main-general}, $T'$ contains a monochromatic path of length at least $\frac{\sqrt{c_1} \: x}{\sqrt{\log x}} \ge \ell$ (here $c_1$ is the constant given by \Cref{thm:main-general}; provided that $c$ is sufficiently large), as claimed.

    \begin{claim}\label{claim:dir-vs-oriented}
        Let $y = |T|/32$, then every induced subgraph $H$ of $T$ on at least $y$ vertices, satisfies $H \rightarrow (P, \dirpath{\ell})$.
    \end{claim}
    
    \begin{proof}
        Let $H$ be as in the claim, and let us assume that $H_R$ does not contain $\dirpath{\ell}$ as a subgraph. This, combined with the above argument, implies that every subset of $x$ vertices of $H$ contains a blue $\dirpath{\ell}$. Furthermore, by Lemma \ref{lem:dfs}, there are disjoint sets $U$ and $W$, such that $|U|=|W| \ge (|H|-\ell)/2 \ge |H|/4$ and all edges from $U$ to $W$ are blue. 
        By pseudorandomness of $T$, it follows that $e_B(H) \ge e(U, W) \ge \eps|U||W| \ge \frac{\eps}{16}y |H| \ge  4(x + n)|H|$ (here we use $x = \eps 2^{-12}|T|$ and $x \ge n$).
        Now, it follows from Lemma \ref{lem:construct-or-path} that $H$ contains a blue copy of $P$, as required.   
    \end{proof}
    
    Without loss of generality, red is the majority colour in $T$. It follows that $e_R(T) \ge |T|(|T|-1)/4 \ge 4(y+n)|T|$.
    By Claim \ref{claim:dir-vs-oriented}, we may assume that every subset of $y$ vertices of $T$ contains a red $\dirpath{\ell}$, because otherwise there is a blue copy of $P$, and we are done. It now follows from Lemma \ref{lem:construct-or-path} that $T$ contains a red copy of $P$, completing the proof of \Cref{thm:oriented-paths-general}.
\end{proof}
    
\section{Concluding remarks and open problems} \label{sec:conclusion}  

    Restating \Cref{thm:main}, we proved that, \whp{}, in every $2$-colouring of a random tournament on $\Omega(n \sqrt{\log n})$ vertices there is a monochromatic path of length $n$. A simple consequence of Theorem \ref{thm:ghrv} generalises the result to $k$ colours, implying the following result.
    
    \begin{thm} \label{thm:k-colours}
        With high probability in any $k$-edge colouring of a random tournament on $\Omega(n^{k-1} \sqrt{\log n})$ vertices, there is a monochromatic path of length $n$. Furthermore, there is a $k$-edge colouring of any tournament on $\frac{1}{2}n^{k-1} (\log n)^{1/k}$ vertices with no monochromatic path of length $n$.
    \end{thm}
    
    \begin{proof}[ sketch]
    Let $T$ be an $(\eps, \sigma \log |T|)$-pseudorandom tournament on $\Omega(n^{k-1}\sqrt{\log n})$ vertices, and consider a $k$-colouring of $T$. Then by \Cref{thm:ghrv}, in colour $k$ there is either a monochromatic path $\dirpath{n}$ or an independent set of size $\Omega(n^{k-2}\sqrt{\log n})$, on which we can use induction, using Theorem \ref{thm:main-general} as the basis. 
    
    For the second part, suppose that $T$ is a tournament on $N = \frac{1}{2^k}n^{k-1} (\log n)^{1/k}$. Following the approach presented in the introduction, we may partition the vertices of $T$ into transitive subtournaments on $\frac{1}{2}\log n$ vertices and a remainder of size at most $\sqrt{n}$. Within the transitive parts, we colour the edges so that there is no monochromatic path of length at least $(\frac{1}{2}\log n)^{1/k}$ (this can be done by generalising the two colours construction for transitive tournaments, also given in the introduction). We colour the edges between the $m$ parts according to a $k$-colouring of the complete graph on $m$ vertices that has no monochromatic directed path of length at least $2m^{1/(k-1)}$ or any monochromatic directed cycle (this can be done using a grid construction; we omit further details). The length of the longest monochromatic path in this colouring is at most $\left(\frac{\log n}{2}\right)^{1/k} \cdot 2\left(\frac{2N}{\log n}\right)^{1/(k-1)} + \sqrt{n} \le n$, as required.
    \end{proof}
    
    The above results give a bound which is best possible up to a polylog factor. It would be very interesting to close this gap.

    Theorem \ref{thm:k-colours} improves the best known upper bound for $k$-colour oriented size Ramsey number of directed paths. Together with the lower bound from \cite{ben2012size} we have the following bounds.
    $$\frac{c_1 n^{2k} (\log n)^{1/k}}{(\log \log n)^{(k+2)/k}} \le r_e\!\left(\dirpath{n},k+1\right) \le  c_2n^{2k} \log n.$$
    Note that for $k = 1$ the gap between the lower and upper bounds is a factor of $O(\log \log n)$, whereas when $k \ge 2$ there is a polylog gap. 

    Theorem \ref{thm:oriented-paths} is a generalisation of Theorem \ref{thm:main} to arbitrarily oriented paths instead of directed paths. In a forthcoming paper, we further generalise Theorem \ref{thm:main} to oriented trees: we show that given any oriented tree $T$ on $n$ vertices, \whp{} in any $2$-colouring of a random tournament $T$ on $cn (\log n)^{5/2}$ vertices, we can find a monochromatic copy of $T$. This bound is tight up to a polylog factor ($n-1$ is a trivial lower bound). It would be interesting to obtain a tight upper bound.

    \textbf{Note added in proof.} After this paper was submitted, two of the authors \cite{benny-shoham} improved on the result of Ben-Eliezer, Krivelevich and Sudakov \cite{ben2012size}. They obtained a lower bound on the oriented size Ramsey number of the directed path that matches, up to a constant factor, the upper bounds that follows from \Cref{thm:main}, thus establishing that $r_e(\dirpath{n}) = \Theta(n^2 \log n)$. They also improve the lower bound for more colours, by showing that the following holds, where $c_3$ is a positive constant.
    $$r_e\left(\dirpath{n}, k+1\right) \ge c_3 n^{2k} (\log n)^{2/(k+1)}.$$
    Note that this still leaves a polylog gap between the lower and upper bounds of the $k$-colour size Ramsey number of the directed path, where $k \ge 3$. It would be interesting to close this gap.

\subsection*{Acknowledgements}
    We would like to thank the anonymous referees for carefully reading an earlier version of this paper and for their valuable remarks and suggestions.

    The second author would like to acknowledge the support of Dr.~Max R\"ossler, the Walter Haefner Foundation and the ETH Zurich
    Foundation.

\providecommand{\bysame}{\leavevmode\hbox to3em{\hrulefill}\thinspace}
\providecommand{\MR}{\relax\ifhmode\unskip\space\fi MR }
\providecommand{\MRhref}[2]{%
  \href{http://www.ams.org/mathscinet-getitem?mr=#1}{#2}
}
\providecommand{\href}[2]{#2}


\begin{thebibliography}{10}

\bibitem{alon-spencer}
N.~Alon and J.~H. Spencer, \emph{The probabilistic method}, fourth ed., Wiley
  Series in Discrete Mathematics and Optimization, John Wiley \& Sons, Inc.,
  Hoboken, NJ, (2016). \MR{3524748}

\bibitem{beck1983size}
J.~Beck, \emph{On size {R}amsey number of paths, trees, and circuits. {I}}, J.
  Graph Theory \textbf{7} (1983), no.~1, 115--129. \MR{693028}

\bibitem{ben2012long-cycles}
I.~Ben-Eliezer, M.~Krivelevich, and B.~Sudakov, \emph{Long cycles in subgraphs
  of (pseudo)random directed graphs}, J. Graph Theory \textbf{70} (2012),
  no.~3, 284--296.

\bibitem{ben2012size}
\bysame, \emph{The size {R}amsey number of a directed path}, J. Combin. Theory
  Ser. B \textbf{102} (2012), no.~3, 743--755. \MR{2900815}

\bibitem{mono-trees-in-tournaments}
M.~{Buci\'c}, S.~{Letzter}, and B.~{Sudakov}.
\newblock \emph{Directed {R}amsey number for trees}.
\newblock {ArXiv e-prints}, (2017).
\newblock submitted, \url{http://arxiv.org/abs/1708.04504}.



\bibitem{erdHos1978size}
P.~Erd\H{o}s, R.~J. Faudree, C.~C. Rousseau, and R.~H. Schelp, \emph{The size
  {R}amsey number}, Period. Math. Hungar. \textbf{9} (1978), no.~1-2, 145--161.
  \MR{479691}

\bibitem{gallai1968directed}
T.~Gallai, \emph{On directed paths and circuits}, Theory of {G}raphs ({P}roc.
  {C}olloq., {T}ihany, 1966) (1968), 115--118. \MR{0233733}

\bibitem{gyarfas1983vertex}
A.~Gy\'arf\'as, \emph{Vertex coverings by monochromatic paths and cycles}, J.
  Graph Theory \textbf{7} (1983), no.~1, 131--135. \MR{693029}

\bibitem{hasse1965algebraischen}
M.~Hasse, \emph{Zur algebraischen {B}egr\"undung der {G}raphentheorie. {I}},
  Math. Nachr. \textbf{28} (1965), 275--290. \MR{0179105}

\bibitem{benny-shoham}
S.~{Letzter} and B.~{Sudakov}.
\newblock \emph{The oriented size Ramsey number of directed paths}.
\newblock {submitted}.


\bibitem{ramsey1929problem}
F.~P. Ramsey, \emph{On a {P}roblem of {F}ormal {L}ogic}, Proc. London Math.
  Soc. \textbf{S2-30} (1929), no.~1, 264. \MR{1576401}

\bibitem{raynaud1973circuit}
H.~Raynaud, \emph{Sur le circuit hamiltonien bi-color\'e dans les graphes
  orient\'es}, Period. Math. Hungar. \textbf{3} (1973), 289--297. \MR{0366739}

\bibitem{reimer2002ramsey}
D.~Reimer, \emph{The {R}amsey size number of dipaths}, Discrete Math.
  \textbf{257} (2002), no.~1, 173--175. \MR{1931501}

\bibitem{roy1967nombre}
B.~Roy, \emph{Nombre chromatique et plus longs chemins d'un graphe}, Rev.
  Fran\c{c}aise Informat. Recherche Op\'erationnelle \textbf{1} (1967), no.~5,
  129--132. \MR{0225683}

\bibitem{vitaver1962determination}
L.~M. Vitaver, \emph{Determination of minimal coloring of vertices of a graph
  by means of {B}oolean powers of the incidence matrix}, Dokl. Akad. Nauk SSSR
  \textbf{147} (1962), 758--759. \MR{0145509}

\end{thebibliography}
\end{document}